\documentclass[12pt]{article}
\usepackage{amssymb,latexsym,amsfonts,amsmath,amsthm,verbatim,eucal,hyperref}

\def\EE{\mathbb{E}}
\def\VV{\mathbb{V}}
\def\RR{\mathbb{R}}

\def\ZZ{\mathbb{Z}}

\def\1{\mathbf{1}}
\def\eps{\varepsilon}

\def\mes{\operatorname{mes}}
\def\supp{\operatorname{supp}}

\theoremstyle{plain}
\newtheorem{thm}{Theorem}
\newtheorem*{thm*}{Theorem}
\newtheorem{lemma}[thm]{Lemma}
\newtheorem*{lemma*}{Lemma*}
\newtheorem{prop}[thm]{Proposition}
\newtheorem*{prop*}{Proposition}

\newtheorem*{cor*}{Corollary}

\newtheorem*{cl*}{Claim}
\theoremstyle{remark}
\newtheorem{rmk}[thm]{Remark}
\newtheorem*{not*}{Notation}
\theoremstyle{definition}

\newtheorem*{qn*}{Question}

\begin{document}

\title{One more proof of the Erd\H{o}s--Tur\'an inequality,\\
and an error estimate in Wigner's law.}
\author{Ohad N.\ Feldheim$^1$, Sasha Sodin$^{1,2}$}
\maketitle

\footnotetext[1]{[ohadf; sodinale]@post.tau.ac.il; address:
School of Mathematical Sciences, Tel Aviv University, Ramat Aviv,
Tel Aviv 69978, Israel}
\footnotetext[2]{Supported in part by the Adams Fellowship Program of the
Israel Academy of Sciences and Humanities and by the ISF.}

Erd\H{o}s and Tur\'an \cite{ET} have proved the following
inequality, which is a quantitative form of Weyl's equidistribution
criterion.

\begin{prop}[Erd\H{o}s -- Tur\'an]\label{p:et.1}
Let $\nu$ be a probability measure on the unit circle $\mathbb{T} = \RR \diagup 2\pi\ZZ$.
Then, for any $n_0 \geq 1$ and any arc $A \subset \mathbb{T}$,
\begin{equation}\label{eq:et.1}
\left| \nu(A) - \frac{\mes A}{2 \pi} \right|
    \leq K_\text{\em\ref{p:et.1}} \left\{ \frac{1}{n_0} + \sum_{n=1}^{n_0} \frac{|\widehat{\nu}(n)|}{n}\right\}~,
\end{equation}
where
\[ \widehat{\nu}(n) = \int_\mathbb{T} \exp(-in\theta) d\nu(\theta)~,\]
and $K_1>0$ is a universal constant.
\end{prop}

A number of proofs have appeared since then, an especially elegant one given
by Ganelius \cite{Ga}. In most of the proofs, the indicator of $A$ is approximated
by its convolution with an appropriate (Fej\'er-type) kernel. We shall present
another proof, based on the arguments developed by Chebyshev, Markov, and Stieltjes to prove
the Central Limit Theorem (see Akhiezer \cite[Ch.~3]{A}). In this approach, the indicator
of $A$ is approximated from above and from below by certain interpolation polynomials.
The argument does not use the group structure on $\mathbb{T}$, and thus works in a
more general setting.

\vspace{2mm}\noindent
In Section~\ref{s:i}, we formulate a slightly different proposition and show that
it implies Proposition~\ref{p:et.1}. In Section~\ref{s:cmt} we reproduce the part
of the arguments of Chebyshev, Markov, and Stieltes that we need for the sequel.
For the convenience of the reader, we try to keep the exposition self-contained.
In Section~\ref{s:pr.1} we apply the construction of Section~\ref{s:cmt} to prove
the Erd\H{o}s--Tur\'an inequality. In Section~\ref{s:w} we formulate another inequality
that can be proved using the same construction. As an application to random matrices,
we use an inequality from \cite{FS} and deduce a form of Wigner's law with a reasonable
error estimate.

\section{Introduction}\label{s:i}

\noindent Let the measure $\sigma_{1}$ on $\RR$ be defined by
\[ d\sigma_1(x) = \frac{1}{\pi} (1 - x^2)_+^{-1/2} \, dx~.\]
Let $T_n(\cos \theta) = \cos n\theta$ be the Chebyshev polynomials of the first kind;
these are orthogonal with respect to $\sigma_1$. We shall prove the Erd\H{o}s -- Tur\'an
inequality in the following form:

\begin{prop}\label{p:et.2}
Let $\mu$ be a probability measure on $\RR$ \footnote{We do not assume that
$\supp \mu \subset [-1, 1]$}. Then, for any $n_0 \geq 1$
and any $x_0 \in \RR$,
\begin{equation}\label{eq:et.2}
\big| \mu[x_0, +\infty) - \sigma_1[x_0, +\infty) \big|
    \leq K_\text{\em\ref{p:et.2}} \left\{ \frac{1}{n_0}
        + \sum_{n=1}^{n_0} \frac{1}{n} \left|\int_\RR T_n(x) d\mu(x)\right| \right\}~.
\end{equation}
\end{prop}

\begin{proof}[Proposition~\ref{p:et.2} implies Proposition~\ref{p:et.1}]
Let $\nu$ be a measure on $\mathbb{T}$, and let $A \subset \mathbb{T}$ be
an arc. Rotate $\mathbb{T}$ (together with $\nu$ and $A$) moving the center of $A$ to $0$; this
does not change the right-hand side of (\ref{eq:et.1}).

Denote $\nu_1(B) = \nu(B) + \nu(-B)$; $\nu_1$ is a measure on $[0, \pi]$. The change of variables
$x = \cos \theta$ pushes it forward to $\mu_1$ on $[-1, 1]$. Now apply Proposition~\ref{p:et.2}
to $\mu_1$, observing that
\[ \int_{-1}^1 T_n(x) d\mu_1(x) = \Re \, \widehat{\nu}(n)~. \]
\end{proof}

\section{The Chebyshev--Markov--Stieltjes construction}\label{s:cmt}
Let $\sigma$ be a probability measure on $\RR$ (with finite moments); let $S_0,S_1,\cdots$
be the orthogonal polynomials with respect to $\sigma$. For a probability measure $\mu$ on $\RR$,
denote
\[ \eps_n = \eps_n(\mu) = \int_\RR S_n(x) d\mu(x)~, \quad n=1,2,3,\cdots~.\]
We shall estimate the distance between $\mu$ and $\sigma$ in terms of the numbers $\eps_n$.

Let $x_1 < x_2 < \cdots < x_{n_0}$ be the zeros of $S_{n_0}$. Construct the polynomials $P,Q$ of
degree $\leq 2n_0 - 2$, so that
\[\begin{split}
    P(x_k) &= \begin{cases}
                0, &1 \leq k < k_0 \\
                1, &k_0 \leq k \leq n_0
             \end{cases};
    \qquad P'(x_k) = 0 \quad \text{for $k \neq k_0$}; \\
    Q(x_k) &= \begin{cases}
                0, &1 \leq k \leq k_0 \\
                1, &k_0 < k \leq n_0
             \end{cases};
    \qquad Q'(x_k) = 0 \quad \text{for $k \neq k_0$}~.
\end{split}\]

\begin{lemma}[Chebyshev--Markov--Stieltjes]
\[ P \geq \mathbf{1}_{\big[x_{k_0}, +\infty\big)} \geq \mathbf{1}_{\big(x_{k_0}, +\infty\big)}\geq Q~. \]
\end{lemma}

\begin{proof}
Let us prove for example the first inequality. The derivative $P'$ of $P$ vanishes at $x_k$,
$k \neq k_0$, and also at intermediate points $x_k < y_k < x_{k+1}$, $k \neq k_0,n_0$. The degree
of $P'$ is at most $2 n_0 - 3$, hence it has no more zeroes.

\noindent Now, $P(x_{k_0}) > P(x_{k_0 - 1})$; hence $P$ is increasing on $(x_{k_0-1}, y_{k_0 +1})$.
Therefore $P'$ is decreasing on $(y_{k_0 + 1}, x_{k_0+2})$, increasing on $(x_{k_0 + 2}, y_{k_0 + 3})$,
et cet. Thus $P(x) \geq 1$ for $x \geq x_{k_0}$. Similarly, $P(x) \geq 0$ for $x < x_{k_0}$.
\end{proof}

Let $P = \sum_{n=0}^{n_0} p_n S_n$, $Q = \sum_{n=0}^{n_0} q_n S_n$. Then
\[\begin{split}
\mu[x_{k_0}, +\infty)
    &\leq \int_\RR P(x) d\mu(x)
    = p_0 + \sum_{n=1}^{2n_0-2} \eps_n p_n \\
    &= q_0 + (p_0 - q_0) + \sum_{n=1}^{2n_0-2} \eps_n p_n \\
    &\leq \sigma(x_{k_0}, +\infty) + (p_0 - q_0) + \sum_{n=1}^{2n_0-2} |\eps_n| |p_n|~.
\end{split}\]
Similarly,
\[ \mu(x_{k_0}, +\infty) \geq \sigma[x_{k_0}, +\infty) - (p_0 - q_0)
    - \sum_{n=1}^{2n_0-2} |\eps_n| |q_n|~.\]
Therefore
\begin{equation}\label{eq:d}
\big| \mu[x_{k_0}, +\infty) - \sigma[x_{k_0}, +\infty) \big|
    \leq (p_0 - q_0) + \sum_{n=1}^{2n_0-2} |\eps_n| \max(|p_n|,|q_n|)~.
\end{equation}
Thus we need to estimate $p_0-q_0$, $|p_n|$, $|q_n|$. This can be done using the following
observation (which we have also used in \cite{me}.) Let $R$ be the Lagrange interpolation
polynomial of degree $n_0 - 1$, defined by
\[ R(x_k) = \delta_{kk_0}~, \quad k = 1,2,\cdots,n_0~.\]
Equivalently,
\begin{equation}\label{eq:r}
R(x) = \frac{S_{n_0}(x)}{S_{n_0}'(x_{k_0}) (x - x_{k_0})}~.
\end{equation}

\begin{lemma} $P-Q=R^2$.
\end{lemma}
\begin{proof}
The polynomial $P-Q$ has multiple zeroes at $x_k$, $k \neq k_0$. Therefore $R^2 \, | \, (P-Q)$.
Also, $\deg R^2 = 2 n_0 - 2 \geq \deg (P-Q)$, and
\[ R^2(x_{k_0}) = 1 = P(x_{k_0}) - Q(x_{k_0})~. \]
\end{proof}

Thus
\begin{equation}\label{eq:pq0}
p_0 - q_0 = \int_\RR R^2(x) d\sigma(x)
\end{equation}
and
\begin{multline}\label{eq:pn}
|p_n| = \left|\int_\RR P(x) S_n(x) d\sigma(x)\right| \\
    \leq \left|\int_{x_{k_0}}^\infty S_n(x) d\sigma(x)\right|
        + \left| \int_\RR (P(x) - \mathbf{1}_{[x_{k_0}, +\infty)}(x)) S_n(x) d\sigma(x) \right| \\
    \leq \left|\int_{x_{k_0}}^\infty S_n(x) d\sigma(x)\right|
        + \int_\RR R^2(x) |S_n(x)| d\sigma(x)~.
\end{multline}
Similarly,
\[ |q_n| \leq |\int_{x_{k_0}}^\infty S_n(x) d\sigma(x)|
        + \int_\RR R^2(x) |S_n(x)| d\sigma(x)~.\]

\section{Proof of Proposition~\ref{p:et.2}}\label{s:pr.1}
We apply the framework of Section~\ref{s:cmt} to $\sigma = \sigma_1$, $S_n = T_n$.
Let $x_{k_0} = \cos \theta_0$, $0 \leq \theta_0 \leq \pi/2$.  Then
\[ T_{n_0}'(\cos \theta_0) \cdot -\sin \theta_0 = - n_0 \sin n \theta_0~,\]
and hence
\[ |T_{n_0}'(x_0)| = \frac{n_0}{|\sin \theta_0|} = \frac{n_0}{\sqrt{1 - x_{k_0}^2}}~.\]
Thus, according to (\ref{eq:pq0}),
\[\begin{split}
p_0 - q_0 &= \int_\RR \frac{T_{n_0}(x)^2}{T_{n_0}'(x_0)^2 (x-x_0)^2} d\sigma_1(x) \\
    &= \frac{\sin^2 \theta_0}{4\pi n_0^2}
        \int_0^\pi \frac{\cos^2 n_0 \theta}
                        {\sin^2 \frac{\theta+\theta_0}{2} \sin^2
                            \frac{\theta-\theta_0}{2}} d\theta~.
\end{split}\]

Now,
\[ \int_0^{\theta_0/2} \leq \int_0^{\theta_0/2} C_1 d\theta / \theta_0^4
    \leq C_1/\theta_0^3 \leq C_2 n_0 / \theta_0^2~, \]
\[ \int_{\theta_0/2}^{\theta_0 - \pi/(3n_0)}
    \leq C_3 \int_{\theta_0/2}^{\theta_0 - \pi/(3n_0)} \frac{d\theta}{\theta_0^2 (\theta-\theta_0)^2}
    \leq \frac{C_4 n_0}{\theta_0^2}~,\]
and similarly
\[ \int_{\theta_0+\pi/(3n_0)}^{\pi} \leq C_5 n_0 / \theta_0^2~.\]
Finally,
\[ |T_{n_0}'(\cos \theta)| = n_0 \frac{|\sin n_0 \theta|}{\sin \theta}
    \geq n_0 / (C_6 \theta_0) \geq |T_{n_0}'(\cos \theta_0)|/C_7\]
for $|\theta - \theta_0| \leq \pi/(3n_0)$, hence
\[ \int_{\theta_0 - \pi/(3n_0)}^{\theta_0+\pi/(3n_0)}
    \frac{T_{n_0}(\cos \theta)^2 d\theta}{T_{n_0}'(\cos \theta_0)^2 (\cos \theta- \cos \theta_0)^2}
        \leq C_8 / n_0~. \]
Therefore
\begin{equation}\label{eq:pq0_1}
p_0 - q_0 \leq C/n_0~.
\end{equation}

\vspace{2mm}\noindent
Next,
\begin{equation}\label{eq:pqn_1.1}
\int_{x_{k_0}}^\infty T_n(x) d\sigma_1(x)
    = \int_0^{\theta_0} \cos n\theta \, \frac{d\theta}{\pi} = \frac{\sin n\theta_0}{n\pi}~;
\end{equation}
\begin{equation*}\begin{split}
\int_\RR R^2(x) |T_n(x)| d\sigma_1(x)
    &= \int_0^\pi \frac{\cos^2 n_0\theta}
                       {\frac{n_0^2}{\sin^2 \theta_0} (\cos \theta - \cos \theta_0)^2}
            |\cos n\theta| \frac{d\theta}{\pi} \\
    &\leq \frac{C_1 \theta_0^2}{n_0^2} \int_0^\pi
        \frac{\cos^2 n_0\theta \, |\cos n\theta| \, d\theta}
             {\sin^2 \frac{\theta+\theta_0}{2} \sin^2 \frac{\theta-\theta_0}{2}}~.
\end{split}\end{equation*}
Now,
\[ \int_0^{\theta_0/2} \leq C_2 /\theta_0^3 \leq C_3 n_0 / \theta_0^2~; \]
\[ \int_{\theta_0/2}^{\theta_0 - \pi/(3n_0)}
    \leq C_4 \int_{\theta_0/2}^{\theta_0 - \pi/(3n_0)} \frac{d\theta}{\theta_0^2 (\theta - \theta_0)^2}
    \leq C_5 n_0 / \theta_0^2~,\]
and similarly
\[ \int_{\theta_0+\pi/(3n_0)}^{\pi} \leq C_6 n_0 / \theta_0^2~;\]
\[ \int_{\theta_0 - \pi/(3n_0)}^{\theta_0 + \pi/(3n_0)} \leq (C_7/n_0) (n_0^2 / \theta_0^2)
    = C_7 n_0/\theta_0^2~.\]
Therefore
\begin{equation}\label{eq:pqn_1.2}
\int_\RR R^2(x) |T_n(x)| d\sigma_1(x) \leq C_8/n_0~.
\end{equation}
Combining (\ref{eq:pn}), (\ref{eq:pqn_1.1}) and (\ref{eq:pqn_1.2}), we deduce:
\begin{equation}\label{eq:pqn_1}
|p_n| \leq C/n~.
\end{equation}
Similarly, $|q_n| \leq C/n$.

\begin{proof}[Proof of Proposition~\ref{p:et.2}]
Substitute (\ref{eq:pq0_1}) and (\ref{eq:pqn_1}) into (\ref{eq:d}), taking
\[ m_0 = \lceil n_0/2 \rceil +1 \]
instead of $n_0$. We deduce that (\ref{eq:et.2}) holds when $x_0 = x_{k_0}$
is a non-negative zero of $T_{m_0}$.  By symmetry, a similar inequality holds
for negative zeroes. For a general $x_0 \in \RR$, apply the inequality to the
two zeroes of $T_{m_0}$ that are adjacent to $x_0$ (one of them may formally
be $\pm \infty$.)
\end{proof}

\section{Another inequality, and an application to Wigner's law}\label{s:w}

Let the measure $\sigma_{2}$ on $\RR$ be defined by
\[ d\sigma_2(x) = \frac{2}{\pi} (1 - x^2)_+^{1/2} \, dx~.\]
Let $U_n(\cos \theta) = \cos n\theta$ be the Chebyshev polynomials of the second kind;
these are orthogonal with respect to $\sigma_2$.

\begin{prop}\label{p:et.w}
Let $\mu$ be a probability measure on $\RR$. Then, for any $n_0 \geq 1$
and any $x_0 \in \RR$,
\begin{multline}\label{eq:et.w}
\big| \mu[x_0, +\infty) - \sigma_2[x_0, +\infty) \big| \\
    \leq K_\text{\em\ref{p:et.w}} \left\{ \frac{\rho(x_0; n_0)}{n_0}
        + \rho(x_0; n_0)^{1/2} \sum_{n=1}^{n_0} n^{-1} \left|\int_\RR U_n(x) d\mu(x)\right| \right\}~,
\end{multline}
where $\rho(x; n_0) = \max(1 - |x|, n_0^{-2})$.
\end{prop}
Observe that $\rho \leq 1$. Similar inequalities with $1$ instead of $\rho$ have been
proved by Grabner \cite{G} and Voit \cite{V}. On the other hand, the dependence on $x$
in (\ref{eq:et.w}) is sharp, in the following sense: for any $x_0$, there exists a probability
measure $\mu$ on $\RR$ such that $\int_\RR U_n(x) d\mu(x) = 0$ for $1 \leq n \leq n_0$, and
\[ \big| \mu[x_0, +\infty) - \sigma_2[x_0, +\infty) \big| \geq C^{-1} \rho(x_0; n_0)/n_0~,\]
where $C>0$ is independent of $n_0$; cf.\ Akhiezer \cite[Ch.~3]{A}.

The proof of Proposition~\ref{p:et.w} is parallel to that of Proposition~\ref{p:et.2}:
we apply the inequalities of Section~\ref{s:cmt} to the measure $\sigma_2$ and the polynomials
$U_n$.

Grabner \cite{G} and Voit \cite{V} have applied their inequalities to estimate the cap
discrepancy of a measure on the sphere. We present an application to random matrices.

\vspace{2mm}\noindent
Let $A$ be an $N \times N$ Hermitian random matrix, such that
\begin{enumerate}
\item $\{ A_{uv} \, | \, 1 \leq u \leq v \leq N \}$ are independent,
\item $\EE |A_{uv}|^{2k} \leq (Ck)^k$, $k=1,2,\cdots$;
\item the distribution of every $A_{uv}$ is symmetric, and $\EE |A_{uv}|^2 = 1$ for $u \neq v$.
\end{enumerate}
Let $\mu_A = N^{-1} \sum_{k=1}^N \delta_{\lambda_k(A)/(2\sqrt{N})}$ be the empirical measure
of the eigenvalues of $A$ (which is a random measure). By \cite[Theorem~1.5.3]{FS},
\[ 0 \leq \EE \int_\RR U_n(x) d\mu_A(x) \leq Cn/N~, \quad 1 \leq n \leq N^{1/3}~.\]
Applying Proposition~\ref{p:et.w}, we deduce the following form of Wigner's law:
\begin{prop} Under the assumptions 1.-3.,
\begin{multline}\label{eq:wl}
\left| \EE \, \# \left\{ k \, \big| \, \lambda_k > 2\sqrt{N} x_0 \right\}
    - N \sigma_2(x_0, +\infty) \right| \\
    \leq C \max\left(N^{2/3}(1-|x_0|), 1\right)
\end{multline}
for any $x_0 \in \RR$.
\end{prop}

Better bounds are available for $x \in (-1 + \eps, 1-\eps)$ (cf.\ G\"otze
and Tikhomirov \cite{GT}, Erd\H{o}s, Schlein, and Yau \cite{ESY}). On the other hand,
for $x$ very close to $\pm 1$, the right-hand side in our bound is of order $O(1)$, which
is in some sense optimal.

\begin{rmk}
A similar method allows to bound the variance of the number of eigenvalues on a half-line:
\[ \VV \# \left\{ k \, \big| \, \lambda_k > 2\sqrt{N} x_0 \right\}
    \leq C \max\left(N^{2/3}(1-|x_0|), 1\right)^{5/2}; \]
therefore one can also bound the probability that
$\# \left\{ k \, \big| \, \lambda_k > 2\sqrt{N} x_0 \right\}$
deviates from $N \sigma_2(x_0, +\infty)$.
\end{rmk}


\begin{thebibliography}{99}

\bibitem{A}         N.~I.~Akhiezer,
                    {\em The classical moment problem and some related
                      questions in analysis},
                    Hafner Publishing Co., New York 1965 x+253 pp.
\bibitem{ESY}       L.~Erd\H{o}s, B.~Schlein, H.-T.~Yau, with an appendix
                        by J.~Bourgain,
                    {\em Local semicircle law and complete delocalization
                        for Wigner random matrices},
                    preprint: arXiv:0803.0542
\bibitem{ET}        P.~Erd\H{o}s, P.Tur\'an,
                    {\em On a problem in the theory of uniform distribution, I-II},
                    Nederl.\ Akad.\ Wetensch., Proc.~51 (1948), 1146--1154, 1262--1269.
\bibitem{FS}        O.~N.~Feldheim, S.~Sodin,
                    {\em A universality result for the smallest eigenvalues of
                        certain sample covariance matrices},
                    preprint: arXiv:0812.1961.
\bibitem{Ga}        T.~Ganelius,
                    {\em Some applications of a lemma on Fourier series},
                    Acad.\ Serbe Sci.\ Publ.\ Inst.\ Math.\ 11 1957 9--18.
\bibitem{GT}        F.~G\"otze, A.~Tikhomirov,
                    {\em Rate of convergence to the semi-circular law},
                    Probab.\ Theory Related Fields~127 (2003), no.~2, 228--276.
\bibitem{G}         P.~Grabner,
                    {\em Erd\H{o}s-Tur\'an type discrepancy bounds},
                    Monatsh.\ Math.\ 111 (1991), no.~2, 127--135.
\bibitem{me}        S.~Sodin,
                    {\em Random matrices, nonbacktracking walks, and
                        orthogonal polynomials},
                    J.\ Math.\ Phys.\ 48 (2007), no.~12.
\bibitem{V}         M.~Voit,
                    {\em Berry-Esseen-type inequalities for ultraspherical expansions},
                    Publ.\ Math.\ Debrecen 54 (1999), no.~1-2, 103--129.

\end{thebibliography}
\end{document}